\documentclass[11pt]{amsart}

\usepackage{eurosym}
\usepackage{amsmath,amsthm,amsfonts,amssymb,amscd,color}
\usepackage{verbatim}
\usepackage{graphicx}
\usepackage{comment}
\usepackage{soul}

\numberwithin{equation}{section}
\theoremstyle{plain}
\newtheorem{thm}{Theorem}[section]
\newtheorem*{thm-no}{Theorem}
\newtheorem{prop}[thm]{Proposition}
\newtheorem{lemma}[thm]{Lemma}
\newtheorem{clly}[thm]{Corollary}
\newtheorem{maintheorem}{Theorem}

\theoremstyle{definition}
\newtheorem{remark}[thm]{Remark}

\newtheorem{defi}[thm]{Definition}

\newtheorem{claim}{Claim}

\newcommand{\Diff}{\operatorname{Diff}}

\def\Diff{{\rm Diff}}
\def\dim{{\rm dim}}

\def\HC{{\rm HC}}
\def\Per{{\rm Per}}

\DeclareMathOperator{\dime}{dim}

\numberwithin{equation}{section}

\def \cB {{\mathcal B}}

\def \cF {{\mathcal F}}

\def \cM {{\mathcal M}}

\def \cO {{\mathcal O}}
\def \cP {{\mathcal P}}
\def \cQ {{\mathcal Q}}
\def \cR {{\mathcal R}}

\def \cU {{\mathcal U}}
\def \cV {{\mathcal V}}

\def\bN{\mathbb{N}}
\def\bZ{\mathbb{Z}}
\def\bR{\mathbb{R}}

\begin{document}
\title[Finite Measures of Maximal Entropy]{Finite Measures of Maximal Entropy for an Open Set of Partially Hyperbolic Diffeomorphisms}
\author[J. Mongez]{Juan Carlos Mongez}
\address{Instituto de Matem\'atica, Universidade Federal do Rio de Janeiro, Cidade Universit\'aria - Ilha do Fund\~ao, Rio de Janeiro 21945-909,  Brazil}
\email{jmongez@im.ufrj.br}

\author[M. Pacifico]{Maria Jose Pacifico}
\address{Instituto de Matem\'atica, Universidade Federal do Rio de Janeiro, Cidade Universit\'aria - Ilha do Fund\~ao, Rio de Janeiro 21945-909,  Brazil}
\email{pacifico@im.ufrj.br}
\date{}

\begin{abstract}
We consider partially hyperbolic diffeomorphisms $f$ with a one-dimensional central direction  such that the unstable entropy is different from the stable entropy. Our main result proves that such maps have a finite number of ergodic measures of maximal entropy. Moreover, any $C^{1+}$ diffeomorphism near $f$ in the $C^1$ topology possesses at most the same number of ergodic measures of maximal entropy. 
These results extend the findings in \cite{BCS22} to arbitrary dimensions and provides
 an open class of non Axiom A  systems of diffeomorphisms  exhibiting a finite number of ergodic measures of maximal entropy.
We believe our technique, essentially distinct from the one in \cite{BCS22}, is robust and may find applications in further contexts.
\end{abstract}

\begin{thanks}
{MJP and JCM were partially supported by CAPES-Finance Code 001.  MJP was partially supported  by CNPq-Brazil Grant No. 302565/2017-5, FAPERJ (CNE) Grant-Brazil No. E-26/202.850/2018(239069), Pronex: E-26/010.001252/2016.}
\end{thanks}
\maketitle


\section{Introduction}

A compelling feature of dynamical systems involves the rate at which the number of distinguishable orbits grows, a quantifiable aspect captured by the concept of 
entropy.
The computation of entropy can be examined through both topological and invariant measure lenses, interconnected by the variational principle: the topological
entropy is the supremum over the metric entropies of all invariant probability measures.
An invariant measure that attains the highest entropy is called a measure of maximal entropy and its study is a natural way of describing the behavior of most of the relevant orbits of a system.

Sinai, {Margulis}, Ruelle and Bowen pioneered the study of measures of maximal entropy in the 1970s. 
Sinai {and Margulis}, the trailblazers in the field, established   the existence and finiteness of maximal entropy measures for Anosov diffeomorphisms \cite{Sin72,Mar70}.
Subsequently, Ruelle and Bowen  extended and broadened this approach to encompass uniformly hyperbolic (Axiom A)  systems \cite{Bow71, Rue68, Rue78}.

In \cite{New89} Newhouse demonstrated that  $C^\infty$ diffeomorphisms on a compact manifold  exhibit {maximal} entropy measures (but finite smoothness does not, according to Misiurewicz \cite{Mis73}).
Newhouse posed the question of whether the {number} of such measures is finite, particularly for diffeomorphisms on surfaces with positive topological entropy  \cite[Problem 2]{New91}.
As the skew product of a linear Anosov diffeomorphism by a rational rotation exhibits an
infinite number of such measures, this question is relevant only for surfaces. 
Recently, Buzzi, Crovisier and Sarig provided a positive answer to this question  in  \cite{BCS22}. 

In the context of partially hyperbolic diffeomorphisms, the existence and uniqueness 
of equilibrium states is not clear at all. 
In three dimensional manifolds, Buzzi
et all   proved in \cite{BFSV11} that a certain derived from Anosov construction due to Ma\~n\'e 
also has this uniqueness property.
In \cite{NY83}, Newhouse and Young have
shown that some partially hyperbolic examples on  $T^4$ have a unique  measure of maximal entropy.
The results in \cite{CFT19} includes the uniqueness of ergodic measure of maximal entropy  for Ma\~n\'e's example in \cite{M78}. 
Recently, in \cite{LSYY}, the authors prove that Shub's example has a unique  measure of maximal entropy. 

For accessible partially hyperbolic diffeomorphisms of 3-dimensional manifolds having compact center leaves,  either there is a unique  measure of maximal entropy,  or there are a finite number of hyperbolic  measures of maximal entropy, see \cite{HHTU12}.
In  \cite{UVY21} the authors characterize the maximal entropy measure for $C^2$ partially hyperbolic systems with circle central bundles by means of suitable finite set of saddle points.

Let us also mention that establishing the existence and, more importantly, the uniqueness of
ergodic maximal entropy measures  has been successful for flows  without any singularity. Among such works, let us mention the works \cite{Bow74, Kn98, BBFS21, BD20, GR23, LSYY}.

To establish the number of maximum entropy measures for flows with singularities is a challenging task. To the best of the authors' knowledge, there is only one article currently available on this subject, \cite{PYY}, where it is proved  that for a sectional-hyperbolic attractor $\Lambda$, in a $C^1$-open and dense family of flows (including the classical Lorenz attractor, \cite{Lo63}), if the point masses at singularities are not equilibrium states, then there exists a unique equilibrium state supported on $\Lambda$. In particular, there exists a unique measure of maximal entropy for the flow restricted to $\Lambda$.  

Back to the diffeomorphism context, we recall that the notion of unstable and stable entropy, denoted as $h^u(f)$ and $h^s(f)$ respectively, {was introduced in \cite{HSX08}, reminiscent of the works by Yomdin and Newhouse (\cite{Yom87, New89}) and studied  in \cite{Yan21, HHW17, Tah21}}. 

In \cite{UVYY24}, it is shown that there is a finite number of maximal 
u-entropy measures for $C^1$ diffeomorphisms that factor over Anosov and has mostly contracting center.  In the same setting as in \cite{UVYY24}, the same authors  prove the exponential decay of correlations for a u-measure of maximal entropy \cite{UVYY}.

In our previous work \cite{MP23}, we delved into the characteristics of unstable and stable entropy.
 By employing the criterion developed in \cite{CT16, PYY22}, we established that for partially hyperbolic systems possessing a unidimensional central bundle, wherein the unstable entropy surpasses the stable entropy and the unstable foliation is minimal, there robustly exists a unique ergodic measure of maximal entropy. 

In this paper, we address the problem of the existence and uniqueness (finiteness) of entropy maximal measures in the partially hyperbolic context. 
We consider partially hyperbolic diffeomorphisms $f$ with a one-dimensional central direction  such that the unstable entropy exceeds the stable entropy. Our main result proves that such maps have a finite number of ergodic measures of maximal entropy. Moreover, any $C^{1+}$ diffeomorphism near $f$ in the $C^1$ topology possesses at most the same number of ergodic measures of maximal entropy. 
These results extend the findings in \cite{BCS22} to higher dimensions and provides
 an open class of non Axiom A  systems of diffeomorphisms  exhibiting a finite number of ergodic measures of maximal entropy.
To be more precise, let $f:M\to M$ be a continuous function on a compact metric set. Denoting by $\cM_f$ the set of $f$-invariant probability measures,  the variational principle  establishes   the following  relation 
\begin{equation}\label{e-variacional-entropia}
    h(f) = \sup\{h_\mu(f) : \mu \in \cM_f\},
\end{equation}
where $h(f)$ is the topological entropy of $f$ and $h_\mu(f)$ is the metric entropy with respect to the invariant measure $\mu$.

\begin{defi}\label{mme}
A measure of maximal entropy (m.m.e.) is a measure
 $\mu \in \cM_f$ such that $h_\mu(f)$ achieves the supremum in equation \eqref{e-variacional-entropia}.
\end{defi}

In the present work,  we withdraw the hypotheses about  the minimality of the unstable foliation in \cite{MP23}  and prove:
\vspace{0.2cm}

{\bf Main Theorem.} Let $f:M\to M$ be a $C^{1+}$ partially hyperbolic diffeomorphism of a compact manifold $M$ with $TM=E^u \oplus E^c \oplus E^s $ and ${\dime}(E^c)=1$. 
If $h^u(f)\neq h^s(f)$, then $f$ has only a  finite number of ergodic measures of maximal entropy. Moreover, there exists a $C^1$ neighborhood $U$ of $f$, such that for any $C^{1+}$ diffeomorphism $g$ within $U$, the number of ergodic measures of maximal entropy of $g$ is not greater than the number of maximal entropy ergodic measures of $f$.

\subsection{Homoclinic classes with high entropy}
Let $f$ be a $C^r$ diffeomorphism on a closed manifold. Consider a hyperbolic periodic orbit of saddle type ${\cO} = \{f^i(x) : i = 0, \ldots, p - 1\}$ such that $p \geq 1$ and $f^p(x)=x$. 

The stable and unstable manifolds of ${\cO}$ are given by 
$$W^s({\cO})=\cup_{y\in {\cO}}W^s(y) \quad \mbox{and} \quad W^u({\cO})=\cup_{y\in {\cO}}W^u(y) \quad \mbox{respectively,}$$
where $W^s(y)=\{x\in M: d(f^n(x),f^n(y))\to 0 \mbox{ when } n\to \infty\}$ and $W^u(y)=\{x\in M: d(f^n(x),f^n(y))\to 0 \mbox{ when } n\to - \infty\}$. These are $C^{ r}$  manifolds. For details see Subsection \ref{subsectionHomoclinicClass}.

Let
$$
{\Per}_h(f) := \{{\cO} : {\cO} \text{ is a hyperbolic periodic orbit of saddle type}\}.
$$

 Given ${\cO}_1, {\cO}_2 \in {\Per}_h(f)$, let $W^s({\cO}_1)\pitchfork W^s(\cO_2)$ denote the collection of transverse intersection points of $W^s({\cO}_1)$ and  $W^u({\cO}_2)$. 
Two ${\cO}_1, {\cO}_2 \in {\Per}_h(f)$ are {\em  homoclinically related} if $W^u(O_i) \pitchfork W^s(O_j) \neq \emptyset$ for $i\neq j$, we then write ${\cO}_1 \sim {\cO}_2$.

The homoclinic class of a hyperbolic periodic orbit ${\cO}$ of saddle type is defined by
$$
{\HC}({\cO}):=\overline{\left\{{\cO}^{\prime} \in \Per_h(f): {\cO}^{\prime} \sim {\cO}\right\}} \text {. }
$$

 An ergodic invariant probability measure $\mu$ is hyperbolic of saddle type, if it has at least one positive Lyapunov exponent, at least one negative Lyapunov exponent, and no zero Lyapunov exponent. 
 When $\mu$ is a hyperbolic measure for $f$ and ${\cO}\in {\Per}_h(f)$ we write $\mu \sim {\cO}$ when $\mu$ is homoclinically related to ${\cO}$, see Subsection \ref{subsectionHyperbolicMeasures},  definitions \ref{Hyperbolicmeasure} and \ref{RelationMuO}.

The next theorem establishes that for a diffeomorphism $f$ as described in the Main Theorem, there exists a finite set of hyperbolic periodic orbits such that   every ergodic measure with high entropy is {homoclinically} related to one of them.   

Since  $h^{u}(f) = h^{s}(f^{\, -1}),$ see Definition \ref{def-u-s-entropia}, from now on, we only consider the case $h^{u}(f) > h^{s}(f)$; the other case follows by considering $f^{\, -1}$.

\begin{maintheorem}\label{MainTheoA}
Let $f:M\to M$ be a $C^{1+}$ partially hyperbolic diffeomorphism of a close manifold $M$ with $TM=E^u \oplus E^c \oplus E^s $ and ${\dime}(E^c)=1$. If $h^u(f)> h^s(f)$, then for any $h(f)\geq a>h^s(f)$ there exists an integer  $K_a>0$ and $K_a$ hyperbolic orbits  ${\cO}_1,\cdots {\cO}_{K_a}$ with $\operatorname{s-index} \, {\dime}(E^s)+1$  such that $h(f_{|HC({\cO}_i)})\geq a$ for every $i=1,\cdots K_a$ and  any  $\mu\in \cM_f^e$ with $h_\mu(g)>a$ must satisfy $\mu \sim {\cO}_i$ for some $1\leq j\leq K_a$.  
\end{maintheorem}

    It is important to note that the above theorem does not imply a spectral decomposition theorem like \cite[Theorem 1]{BCS22}. Our main difficulty arises from the potential existence of two hyperbolic periodic orbits that are not homoclinically related, but 
    their homoclinic classes are such that one is contained in the other.
    Consequently, it is impossible to make any statement about the entropy of the intersection between two homoclinic classes in Theorem \ref{MainTheoA}. 
    
    { Buzzi, Crovisier and Sarig in their notable work in \cite{BCS22} establishes that $C^\infty$ surface diffeomorphisms with positive topological entropy have at most a finite number of ergodic measures of maximum entropy in general. 
The approach there is strongly supported by the two-dimensionality  to bind topological disks ($su$-quadrilaterals) by two stable and two unstable segments.
These $su$-quadrilaterals {allow} them to obtain topological intersections between stable and unstable manifolds of  two different  measures with large entropy that are weak$^*$ close. 
Using  a dynamic version of a lemma by Sard, they conclude that some of these intersections are transversal leading to a homoclinic relation between the measures.
Therefore, replicating this strategy in a higher-dimensional context is unfeasible. 
We, instead, propose an alternative method in which we define a set where each of its points has a stable Pesin manifold of uniform size.
Taking advantage of Pliss's Lemma, we then prove that  the measure of this set for measures with high entropy is strictly positive. 
Then we were able to prove that for nearby points in this set  their stable and unstable manifolds   
intersect transversaly. This implies
that the number of homoclinic classes with measures with large entropy is finite. And then, using \cite[Corollary 2.14]{BCS22}, we conclude that 
the number of ergodic maximal entropy measures of $f$ is finite.
We believe that our method can be applied to  other higher dimensional contexts.}

    In general, it does not hold that every diffeomorphism  close  to a diffeomorphism with a finite number of measures of maximal entropy has also a finite number of measures of maximal entropy. 
    But for diffeomorphisms as in the theorem above we prove  that every $C^{1+}$ diffeomorphism  sufficiently close to $f$ also has only a finite number of ergodic m.m.e., see Corollary \ref{Finite-MME}.

\subsection{Dependence of the number of ergodic measures of maximal entropy  on the diffeomorphism}

A natural question is  the dependence of the simplex of m.m.e.\ on the diffeomorphism.
For $C^\infty$ surface diffeomorphisms, it is proved in \cite{BCS22} that the number of 
ergodic m.m.e.\ varies semicontinuously on the diffeomorphism. 
Here we achieve the same property for $C^{1+}$ partially hyperbolic diffeomorphisms with a central bundle of dimension one and with unstable entropy bigger than the stable entropy. To announce precisely our result, let $N_{max}(f)$ be the number of ergodic m.m.e.\ of a diffeomorphism $f$. 
    
    \begin{maintheorem}\label{MainTheoB}
    Let $f:M\to M$ be a $C^{1+}$ partially hyperbolic diffeomorphism of a compact manifold $M$ with $TM=E^u \oplus E^c \oplus E^s $ and $\operatorname{dim}(E^c)=1$. If $h^u(f) > h^s(f)$  
    then the function $N_{max} : \operatorname{Diff}^{1+}(M) \to \mathbb{N} \cup \{\infty\}$ is upper semicontinuous in the $C^1$ topology, that is, $\limsup_{g \to f} N_{max}(g) \leq N_{max}(f)$.
    \end{maintheorem}

   {
   The meaning of this theorem lies in resolving a long-standing question within smooth dynamical systems. Although Bowen established the robustness of the number of ergodic measures of maximum entropy for Axiom A diffeomorphisms, the search for non-Axiom A examples remained an open challenge. 
   This theorem provides 
   a non-Axiom A system that maintains the property of finiteness of ergodic measures of maximal entropy under small perturbations.  In addition,  the quantity of such measures  of these perturbations are bounded  by the number of ergodic maximal  entropy measures of $f$.
  The proof of this theorem involves a refined analysis of the properties of unstable and stable entropy, as well as certain estimates related to a set of points exhibiting a stable manifold of uniform size.}

\subsection{Central Lyapunov exponent far from zero}
 
 Let $f:M\to M$ be a $C^{1+}$ partially hyperbolic diffeomorphism of a compact manifold $M$ with $TM=E^u \oplus E^c \oplus E^s $ and $\operatorname{dim}(E^c)=1$. Consider the sets 
$$\cM^+_f:=\{\mu \in \cM_f^e: \lambda_c(f,\mu)\geq 0\} \quad \mbox{ and }\quad \cM^-_f:=\{\mu \in \cM_f^e: \lambda_c(f,\mu)\leq 0\} $$
and let $$h^-(f):=\sup_{\mu \in \cM_f^-}\{h_\mu(f)\} \quad \mbox{ and } \quad h^+(f):=\sup_{\mu \in \cM_f^+}\{h_\mu(f)\}.$$
In \cite{CT21}, it was proved that a $C^{1+}$ partially hyperbolic diffeomorphism of a compact manifold $M$ with one-dimensional center, dense stable and unstable foliations, and $h^+(f)\neq h^-(f)$ has a unique measure of maximal entropy. We prove the following theorem.

\begin{maintheorem}\label{MainCentralLyapunovexponentfarfrom zero}
Let $f:M\to M$ be a $C^{1+}$ partially hyperbolic diffeomorphism of a compact manifold $M$ with $TM=E^u \oplus E^c \oplus E^s $ and $\operatorname{dim}(E^c)=1$. If $h^+(f)\neq h^-(f)$ then $f$ has only a finite number of ergodic maximal entropy.
\end{maintheorem}
The condition $h^-(f)> h^+(f)$ is equivalent to requiring that every ergodic measure with large entropy has a central Lyapunov exponent far from zero, as stated in Proposition \ref{farfromzero}.

When $f$ satisfies the conditions of Theorem \ref{MainTheoA}, we can conclude that $h^-(f)>h^+(f)$. However, it is not clear whether $h^-(f)>h^+(f)$ implies $h^u(f)>h^s(f)$. Therefore, we cannot obtain information about diffeomorphisms close to $f$ when $f$ satisfies $h^-(f)> h^+(f)$.

    \subsection{Equilibrium states}
     Let $\phi:M\to \mathbb{R}$ be a continuous function. An $f-$invariant probability measure $\mu$ is  an {\em equilibrium state} for the potential $\phi$, if $h(f,\mu)+\int \phi d\mu=P(f,\phi)$ where

$$P(f,\phi):=\sup_{\mu\in \cM_f}\{h_\mu(f)+\int \phi d\mu\}.$$ 

The next theorem gives information about equilibrium states for potentials with a certain regularity.

\begin{maintheorem}\label{MainTheoC}
Let $f:M\to M$ be a $C^{1+}$ partially hyperbolic diffeomorphism of a compact manifold $M$ with $TM=E^u \oplus E^c \oplus E^s $, ${\dime}(E^c)=1$ and $\phi$ a  H\"older continuous potential. If green $h^u(f)>\sup \phi\geq \inf \phi >h^s(f) $, then for any $P(f,\phi)\geq a>h^s(f)+\sup \phi$ there exists  an integer $K_a>0$ and $K_a$ hyperbolic orbits  $\mathcal{O}_1,\cdots \mathcal{O}_{K_a}$ with $s-$index ${\dime}(E^s)+1$  such that $P(f_{|HC({\cO}_i)},\phi_{|HC({\cO}_i)})\geq  a$ for every $i=1,\cdots K_a$ and  any  $\mu\in \cM_f^e$ with $P_\mu(f,\phi)\geq a$ must satisfy  $\mu \sim \mathcal{O}_i$ for some $1\leq j\leq K_a$.  
\end{maintheorem}

 As a consequence of Theorem \ref{MainTheoC},   we get that $C^{1+}$ diffeomorphisms in a $C^1$ open neighborhood of $f$ also have a finite number of equilibrium states, see Corollary \ref{finiteEE}.

  Let $M_{\text{max}}(f, \varphi)$ be the number of ergodic equilibrium states of a diffeomorphism $f$ and a potential $\varphi$. 
  Our last result establishes  the following upper semicontinuity property for 
  $M_{\text{max}}(\cdot,\varphi)$.

\begin{maintheorem}\label{MainTheoD}
    Let $f:M\to M$ be a $C^{1+}$ partially hyperbolic diffeomorphism of a compact manifold $M$ with $TM=E^u \oplus E^c \oplus E^s $ and ${\dime}(E^c)=1$, and $\phi$ a H\"older continuous potential.
    If $h^u(f)>\sup \phi\geq \inf \phi >h^s(f) $, then the function $M_{max}(.,\phi) : {\Diff}^{1+}(M) \to {\bN} \cup \{\infty\}$ varies upper semi-continuously in the $C^1$ topology, that is, $\limsup_{g \to f} M_{max}(g) \leq M_{max}(f)$.
\end{maintheorem}

\subsection{Sketch of the proof of the Main Theorem}
The proof establishing the finiteness of the number of ergodic measures of maximal entropy  in the Main Theorem is structured into three key steps:

    
\begin{itemize}

    \item[(1)] Utilizing the assumption $h^u(f)>h^s(f)$, we established that every ergodic measure with substantial entropy possesses a negative central Lyapunov exponent that remains uniformly away from zero. That is, if $h_\mu(f)>h^u(f)$ then 
    $\lambda_c(f,\mu)\leq h_\mu(f)-h^u(f),$
    where $\lambda_c(f,\mu)$ denotes the central Lyapunov exponent with respect to $\mu$. 
    
    \item[(2)] {Consider $0<\rho<1$ and  the set} 
    $$\Lambda_\rho(f):=\{x\in M: \forall n\geq 1 \mbox{ it holds } \|Df^n(x)_{|E^c}\|\leq \rho^n \}.$$    
 
{For every $\rho$ big enough, the set is not empty  and every $x \in \Lambda_\rho(f)$ has  invariant Pesin manifolds of uniform size. 
  In particular, for two points in $\Lambda_\rho(f)$ sufficiently close, their invariant manifolds have transversal intersections.}

    \item[3)] Utilizing the classical Pliss Lemma, 
  we prove that for certain values of $\rho$ (depending on $\lambda_c$), $\Lambda_\rho(f)$ has large measure for all invariant measure with large enough entropy, the set $\Lambda_\rho(f)$ has uniform positive measure with respect to this measure depending on how big is the entropy.  
\end{itemize}

 Items (1), (2) and (3) imply that there is only a finite number of homoclinic classes in the set of hyperbolic measures containing measures of maximal entropy (we are referencing  the equivalence relation described in Definition \ref{def-HR-measures}). 
 By Theorem \ref{criterio-MME}, each of these equivalence  classes possesses at most one measure of maximal entropy. Therefore, the number of maximal entropy ergodic measures for $f$ is finite.

To prove that diffeomorphisms close enough to $f$ have at most the same number of ergodic m.m.e.\ of $f$ we proceed as follows: 

\begin{itemize}
    \item[(4)]Thanks to Lemma \ref{h^u(g)>h^s(g)}  the property $h^u(f)>h^s(f)$ is an open property in $C^{1+}$ diffeomorphisms on $M$, that is, for any $h^u(f) > a > b > h^s(f)$, there exists a $C^1$ neighborhood of $f$ such that every $C^{1+}$ diffeomorphism has unstable entropy greater than $a$ and stable entropy less than $b$.

    \item[(5)] Item (4) allows us to demonstrate that the properties mentioned in Items (2) and (3) are open. Specifically, there exists a $\beta > 0$ such that for every $C^{1+}$ diffeomorphism $g$ within a certain neighborhood of $f$, the set $\Lambda_\rho(g)$ has a measure greater than $\beta$ for each ergodic measure of maximal entropy of $g$. 
    
    \item[(6)] To conclude the proof, we argue by contradiction. 
    Assuming that there is a sequence of $C^{1+}$ diffeomorphisms 
    $g_n$    having a  number of ergodic measures of maximal entropy greater than
    the number of ergodic measures of maximal entropy of $f$ which converges to $f$.
    For each $g_n$ we take an ergodic measure $\nu_n$ of maximal entropy. There is no loss of generality assuming that $\nu_n$ converges to an $f$-invariant measure $\nu$.
    
\item[(7)] Using the consequences from our hypotheses  $h^u(f)> h^s(f)$, 
we prove that $\nu$ is a measure of maximal entropy of $f$.

\item[(8)]
We further get that $\nu$ has an ergodic component that we associate  with $\nu_n$.
  In this way   we  establish an injective correspondence between the ergodic measure of maximal entropy of $g_n$ into the ergodic measure of maximal entropy of $f$. 
  Thus the number of ergodic maximal entropy measures of $g_n$ is bounded by
  the number of ergodic maximal entropy measures of $f$, contradicting the hypotheses.
    \end{itemize}

  \subsection*{Organization of the paper} 
 In Section \ref{s-preliminar} we recall the main previous results 
that we will use. 
In section \ref{s-consequences-hu>hs} we explore the consequences of the hypotheses $h^u(f) > h^s(f)$. The section \ref{Hyperbolic-time} is dedicated to proving that the set of points with Pesin manifolds of uniform size  has a uniform positive size for measures with large entropy. In sections \ref{Proof-A-B}  to \ref{Proof-C-D} we prove the results of our work.
\section{Preliminaries}\label{s-preliminar}

\subsection{Hyperbolic sets and homoclinic class}\label{subsectionHomoclinicClass}
 We will review the definition of homoclinic classes by Newhouse \cite{New72} and some of their properties.

Let $f:M \to M$ be a diffeomorphism on a compact manifold $M$, and let $\Lambda \subset M$ be a compact and invariant set. We say that $\Lambda$ is a hyperbolic set for $f$ if there exists a splitting of the tangent bundle $T\Lambda$ into two bundles $E^s$ and $E^u$ such that for all $x \in \Lambda$,  $Df_xE^s_{x} = E^s_{f(x)}$
and  $Df_xE^u_{x}=$  
 $E^u_{f(x)}$ and there exist a constant $C>0$ and $0 < \lambda < 1$ such that for all $x \in \Lambda$ it holds
$$\left\| Df^n_xv \right\| \leq C\lambda^n \left\| v \right\| \mbox{ for all } v \in E^s_x \mbox{ and } n\geq 1$$
and 
$$\left\| Df^{-n}_xv \right\| \leq C \lambda^{n} \left\| v \right\| \mbox{for all } v \in E^u_x \mbox{ and } n\geq 1.$$ 
It is well known that if $\Lambda$ is a hyperbolic set, then for any $x \in \Lambda$, the following sets are $C^1$-immersed submanifolds \cite{HPS70}.
$$
\begin{array}{l}
W^s(x)=\left\{y \in M: \lim_{n\to \infty}d\left(f^n(x), f^n(y)\right)=0\right\}, \\
W^u(x)=\left\{y \in M: \lim_{n\to \infty} d\left(f^{-n}(x), f^{-n}(y)\right)= 0 \right\} .
\end{array}
$$
The sets $W^{s/u}(x)$ are called the stable and unstable manifolds of $x$ respectively, and for $\varepsilon$ small,  we denote by $W^{s/u}_{\varepsilon}$ the connected component of $W^{s/u}\cap B_\varepsilon(x)$ that contains $x$.
When a periodic point is  hyperbolic   we say that  its orbit is a hyperbolic periodic orbit. We denote the set of all hyperbolic periodic orbits of $f$ by $\operatorname{Per}_h(f)$. The transverse
intersection of $C^r$-submanifolds $U, V \subset M$ is
$$U\pitchfork V=\{x\in U\cap V: T_xU \oplus T_xV=T_x M\}.$$

Given two orbits ${\cO}_1, {\cO}_2 \in {\Per}_h(f)$ we say that they are homoclinically related if
$$
W^s\left({\cO}_1\right) \pitchfork W^u\left({\cO}_2\right) \neq \emptyset \text { and } W^s\left({\cO}_2\right) \pitchfork W^u\left({\cO}_1\right) \neq \emptyset \text {. }
$$
If two orbits, ${\cO}_1$ and ${\cO}_2$, are homoclinically related we write ${\cO}_1 \sim {\cO}_2$. The pair $(\Per_h(f),\sim)$ defines an  equivalence relation. 

\begin{defi}\label{Characterization-Homoclinic-class}
    The homoclinic class of a hyperbolic periodic orbit ${\cO}$ is 
$$
{\HC}({\cO}):=\overline{\left\{{\cO}^{\prime} \in {\Per}_h(f): {\cO}^{\prime} \sim {\cO}\right\}} \doteq
\overline{W^s\left({\cO}\right) \pitchfork W^u\left({\cO}\right)}{.}
$$
\end{defi}

The class $\HC(\cO)$ is called trivial when it coincides with $\cO$.
See \cite{New72} for the proof of $\doteq$. The compact set $\HC(\cO)$ is transitive \cite{New72}. In general, different homoclinic classes may have non-empty intersections.

A hyperbolic set $\Lambda$ is called locally maximal if it possesses an open neighborhood $V$ satisfying $\Lambda = \bigcap_{n\in {\bZ}} f^n(V)$. A fundamental property of hyperbolic sets is the existence of a $\delta > 0$ such that, for any $x, y \in \Lambda$ with $d(x, y) < \delta$, the intersection $W_\varepsilon^s(x) \pitchfork W_\varepsilon^u(y)$ is a unique point $\ z \in M$. When this point $z$ belongs to $\Lambda$, we refer to $\Lambda$ as possessing a local product structure. 
It is known that a hyperbolic set $\Lambda$ is locally maximal if and only if it presents this property of local product structure \cite{Shub13}. 

A basic set $\Lambda$ for $f$ is an $f$-invariant compact set which is transitive, hyperbolic, and locally maximal. A totally disconnected and infinite basic set is referred to as a horseshoe. 
It is known that every homoclinic class can be approximated by horseshoes, see \cite{KH95}.

\subsection{Hyperbolic measures and homoclinic classes}\label{subsectionHyperbolicMeasures} We will recall some of the results presented in \cite{BCS22}.
Let $x$ be a point of $M$. For a nonzero vector $v \in T_xM$, the Lyapunov exponent of $f$ $x$ on direction  $v$ is defined by
$$
\chi(x,v): = \lim_{{n \to +\infty}} \frac{1}{n} \log \left\| Df^n(v) \right\|
$$
when the limit exists.

Let $\cM_f^e$ be the set of  ergodic measures  of $f$. If $\mu \in \cM_f^e$, the Oseledets theorem establishes that for almost every $x \in M$ with respect to $\mu$, the Lyapunov exponent of every nonzero vector $v \in T_xM$ exist \cite{Ose68}.

\begin{thm}[Oseledets]
 Let $(f, \mu)$ be as described above. Then for $\mu$ almost every point, there are real numbers $\chi_1(f,\mu) < \chi_2(f,\mu) < \cdots < \chi_k(f,\mu)$, and a splitting
$$
T_xM = E_{x,1} \oplus E_{x,2} \oplus \cdots \oplus E_{x,k}
$$
satisfying:
\begin{enumerate}
    \item $Df(E_{x,i}) = E_{f(x),i}$ for $i = 1, \ldots, k$.
    \item  For all nonzero $v \in E_{x,i}$,
    $$
    \chi(x,v) = \lim_{{n \to \pm\infty}} \frac{1}{n} \log \left\| Df^n_v \right\| = \chi_i(f,\mu).
    $$
\end{enumerate}

\end{thm}

{The set of points for which the Oseledets Theorem holds for a certain ergodic measure is called a Lyapunov regular set, denoted by ${\cR}$.}

 \begin{defi}\label{Hyperbolicmeasure}
     An ergodic measure $\mu$ is hyperbolic of saddle type if \linebreak $|\chi_i(f,\mu)|>0 $ for every $i$ and    $\chi_1(f, \mu)<0<\chi_{k}(f, \mu)$. We will denote the set of these ergodic measures by $\cM_f^h$.
 \end{defi}
Note that we  consider that every hyperbolic measure is an ergodic measure. 
 
 When $\mu \in \cM_f^h$ we define, for $\mu$-almost every $x\in M$, the splitting  $T_x M=E^+_x(f) \oplus E^-_x(f)$ where $E^+_x,\, E^-_x$ are linear spaces such that \linebreak $\lim _{n \rightarrow \infty} \frac{1}{n} \log \left\|D f_x^n v\right\|<0$ on $E^+(x) \backslash\{0\}$, and $\lim _{n \rightarrow \infty} \frac{1}{n} \log \left\|D f_x^{-n} v\right\|>0$ on $E^-(x) \backslash\{0\}$.

For $f\in \Diff^{1+}(M)$ the stable Pesin set of the point $x \in M$ is
$$
W^{-}(x)=\left\{y \in M: \limsup _{n \to+\infty} \frac{1}{n} \log d\left(f^n(x), f^n(y)\right)<0\right\}.
$$
Similarly one defines the unstable Pesin set as
$$
W^{+}(x)=\left\{y \in M: \limsup _{n \rightarrow+\infty} \frac{1}{n} \log d\left(f^{-n}(x), f^{-n}(y)\right)<0\right\}.
$$
It is known that given $\mu \in \cM_f^h$, for $\mu$ almost every $x\in M$ the stable and unstable Pesin sets  are immersed submanifolds with the dimension ${\dime}(E^-)$ and ${\dime}(E^+)$ respectively and therefore are called  stable and unstable Pesin manifolds \cite{Pes76}. {When $\{x,f(x)\cdots f^{p-1}(x)\}\in \operatorname{Per}_h(f)$ we have that $W^s(f^i(x))=W^-(f^i(x))$}. 
Given a hyperbolic measure $\mu$ of $f$ we define  the  $\operatorname{s-index}$ of $\mu$ as {\em $\operatorname{s-index}(\mu):=\operatorname{dim}(E^-)$.}  
In the remaining of the section, we consider only ergodic measures with the same s-index.

\begin{defi}
For two  measures $\mu_1, \mu_2 \in \cM_f^h$, we write $\mu_1 \preceq \mu_2$ if there are two sets $\Lambda_1, \Lambda_2$ such that $\mu_i\left(\Lambda_i\right)>0$ with the property that for any point $\left(x, y\right) \in \Lambda_1 \times \Lambda_2$ we have that $W^{+}\left(x\right) \pitchfork W^{-}\left(y\right)\neq \emptyset$.
\end{defi} 

\begin{defi}\label{def-HR-measures}
    Two ergodic measures $\mu_1, \mu_2 \in \cM_f^h$ are homoclinically related if $\mu_1 \preceq \mu_2$ and $\mu_2 \preceq \mu_1$. In this case, we write $\mu_1 \sim \mu_2$. The homoclinic class of the measure $\mu\in \cM_f^h$ is the set
    $$HC(\mu)=\{\nu\in \cM_f^h: \mu \sim \nu\},$$
    and the topological homoclinic class is the set  
    $$HC_T(\mu):=\overline{\bigcup_{\nu \in HC(\mu)}\operatorname{supp}(\nu)}.$$
\end{defi} 

The relation $\left(\cM_f^h,\sim\right)$ is an equivalence relation \cite[Proposition 2.11]{BCS22}.

\begin{defi}\label{RelationMuO}
Given $\mu\in \cM_f^h$ and  $\mathcal{O}\in \operatorname{Per}_h(f)$ with $\mathcal{O}=\{p,\cdots f^{n-1}(p)=p\}$. We say that $\mu$ and ${\cO}$ are homoclinically related and write   ${\cO} \sim \mu$ or $p\sim \mu $ if the measure $\frac{1}{n} \sum_{j=0}^{n} \delta_{f^j(p)}$ is homoclinically related to the measure $\mu$. 
    
\end{defi}

\begin{prop}\cite[Corollary 2.14.]{BCS22}\label{HoclinicClassRelation}
    For any $f \in {\Diff}^r(M)$, $r > 1$, and any $\mu \in \cM_f^h$, the set of measures supported by periodic orbits is dense in the set of hyperbolic ergodic measures homoclinically related to $\mu$, endowed with the weak-$\ast$ topology. In particular, there exists ${\cO} \in \operatorname{Per}^h(f)$ such that ${\cO} \sim \mu$ and

$$HC_T(\mu) = HC(\mathcal{O}).$$
\end{prop}

\begin{prop}\label{criterio-MME}
  Let $f: M \rightarrow M$ be a $C^r$ diffeomorphism for some  $r>1, \phi: M \rightarrow {\bR}$ be a H\"older continuous potential, and $p$ a hyperbolic periodic point with orbit ${\cO}$. Then there is at most one hyperbolic equilibrium state for $\phi$ which is homoclinically related to ${\cO}$, and its support coincides with ${\HC}({\cO})$.
  
\end{prop}
\begin{proof}
    The proof is a consequence of \cite{Ova18}, which is a generalization of \cite{Sar13}, and Theorem \cite[Theorem 3.1]{BCS22}.
\end{proof}

\begin{remark}
{It is important to observe that the above proposition does not imply that for each $\mathcal{O}\in \operatorname{Per}(f)$, $HC(\mathcal{O})$ has at most one hyperbolic measure of maximal entropy. There can exist two orbits $\mathcal{O}_1$ and $\mathcal{O}_2 \in \operatorname{Per}(f)$ that are not homoclinically related to each other, each one related to a hyperbolic measure for maximal entropy, such that $HC(\mathcal{O}_1)\subset HC(\mathcal{O}_2)$. Consequently, $HC(\mathcal{O}_2)$ could have two ergodic measures of maximal entropy.}
\end{remark}
\subsection{Partially hyperbolic systems}\label{ss-PH}
The notion of partially hyperbolic systems represents a natural extension of the concept of uniformly hyperbolic systems. 

\begin{defi}\label{def-partialh}
     A diffeomorphism $f$ defined on a compact manifold $M$ is  partially hyperbolic if it  admits a non-trivial $Df$- invariant splitting of the tangent bundle $TM=E^s\oplus E^c\oplus E^u$, such that, all unit vectors $v^{\sigma}\in E^\sigma_x$ ($\sigma=s,c,u$) with $x\in M$ satisfy

\begin{equation*}
    \|Df_xv^s\|<\|Df_xv^c\|<\|Df_xv^u\|
\end{equation*}
for some appropriate Riemannian metric; $f$ must also satisfy that $0<\|Df_{|E^s}\|<\xi<1$ and $0<\|Df^{-1}_{|E^u}\|<\xi<1$.

\end{defi}
\subsubsection{ Existence of equilibrium states}\label{subsebtion-dimEc1-mme}
Newhouse proved that every $C^\infty$ diffeomorphism has m.m.e.\ \cite{New91}, but this is not true for diffeomorphisms with less differentiability \cite{Mis73}. It is known that every Anosov diffeomorphism is expansive and therefore has equilibrium states for every H\"older continuous potential. In general, a partially hyperbolic system does not have equilibrium states.
   
   Let $\phi$ be a H\"older continuous potential. Consider the function  $P_{metric}(f,\phi)$,  defined as 
 $$P_{metric}(f, \phi):  \cM_f \to \bR, \,\, \mu \mapsto P_{\mu}(f, \phi).
 $$

Since $\cM_f$ is a compact set in the weak * topology, if the metric entropy  is upper semicontinuous then $P_{metric}(f,\phi)$ achieves a maximal value and so $(f, \phi)$ has an equilibrium state. 
 
 For homeomorphisms, Bowen showed \cite{Bow72} that the metric entropy
 is upper semicontinuous whenever $(M, f)$ is a $h$-expansive system. 
In particular, the pressure $P_\mu(f,\phi)$ is also upper semicontinuous  and then
$(f, \phi)$ admits an equilibrium state for any continuous potential $\phi$.
In particular,  since hyperbolic systems with one-dimensional center bundles are $h$-expansive,
 \cite{CY05, DFPV12, LVY13},  they have equilibrium states.

\subsection{Unstable entropy }

In this section, we recall the notions of unstable $h^u(f)$ and stable $h^s(f)$  topological entropy for a  partially hyperbolic diffeomorphism $f$, see \cite{Yan21, HHW17,Tah21,LY85a,LY85b}.

\subsubsection{Unstable metric entropy}
Here we follow closely  \cite{Tah21}. 

    Consider a compact metric space $M$ with the Borel $\sigma$-algebra ${\cB}$ and let $\mu$ represent a Borel probability measure defined on this space. Let ${\cP} = \{P_i\}_{i\in I}$ be a partition of $M$ into measurable subsets, i.e., $M = \bigcup_{P_i\in\mathcal{P}} P_i$ such that $P_i\in {\cB}$, $\mu(P_i \cap P_j) = 0$, and $\mu\left(\bigcup_{P_i\in {\cP}} P_i\right) = \mu(M) = 1$. We denote by $\sigma({\cP})$ the smallest $\sigma$-algebra containing all $P_i\in {\cP}$.

Given two partitions ${\cP}$ and ${\cQ}$, we use the notation ${\cP} \leq {\cQ}$ if   $\sigma({\cP})\subset\sigma({\cQ})$.

\begin{defi}
    We say that a partition ${\cP}$ is measurable (or countably generated) with respect to $\mu$ if there exists a {family of measurable sets} $\{A_i\}_{i\in\mathbb{N}}$ and a measurable set $F$ of full measure such that if $B \in {\cP}$, then there exists a sequence $\{B_i\}$, where $B_i \in \{A_i, A_i^c\}$ such that $B \cap F = \bigcap_i (B_i \cap F)$.
\end{defi}
Consider $\pi: M \rightarrow M/\mathcal{P}$ as the canonical projection, mapping each point in $M$ to the corresponding element in the quotient space formed by  the atoms of the partition. 
This naturally gives rise to a measurable structure on $M/\mathcal{P}$ via  the measure $\widetilde{\mu} = \pi^*\mu$ defining it on the associated pushed $\sigma$-algebra. 
Additionally, for any point $x$ in $M$, $P(x)$ represents the element of the partition to which $x$ belongs.

\begin{defi}\label{ConditionalMeasure}
Given a partition $\mathcal{P}$, a family $\{\mu_P\}_{P \in \mathcal{P}}$ is a system of conditional measures for $\mu$ (with respect to $\mathcal{P}$) if, for a given $\varphi \in C_0(M)$:
\begin{enumerate}
    \item $P \mapsto \int \varphi \, d\mu_P$ is $\tilde{\mu}$-measurable.
    \item $\mu_P(P) = 1$ $\tilde{\mu}$-a.e.
    \item $\int_M \varphi \, d\mu = \int_{M/P} \left(\int_{P} \varphi \, d\mu_P\right)d\tilde{\mu} \,  = \int_M \left(\int_{P(x)} \varphi \, d\mu_P(x)\right) \, d\mu$.
\end{enumerate}

\end{defi}

\begin{thm} \cite[Rokhlin's Disintegration]{Roh49}
Let $P$ be a measurable partition of a compact metric space $M$, and $\mu$ a Borel probability measure. Then there exists a unique disintegration by conditional probability measures for $\mu$.
\end{thm}

Let $f : M \rightarrow M$ be a diffeomorphism of a compact manifold preserving a probability measure $\mu$. Consider a partition $P = \{P_1, P_2, \ldots, P_n\}$ into measurable subsets. Then:

    $$H_\mu(\mathcal{P}) = \sum_{i=1}^{n} -\mu(P_i) \log(\mu(P_i)) = \int_{M} -\log(\mu(P(x))) \, d\mu,$$
    where $P(x) = P_i$ for $x \in P_i$. And
    
     $$h_\mu(\mathcal{P}, f) = \lim_{n\to\infty} \frac{1}{n} H_\mu(\mathcal{P}_n) = \lim_{n\to\infty} \left(H_\mu(\mathcal{P}_n) - H_\mu(\mathcal{P}_{n-1})\right),$$ where $\mathcal{P}_n = \mathcal{P} \vee f^{-1}(\mathcal{P}) \vee \ldots \vee f^{-n+1}(\mathcal{P})$ is a refinement of $\mathcal{P}$ by the action of $f$.

Now let ${\cF}$ be an $f$-invariant foliation, that is, {$f({\cF}(x))={\cF}(f(x))$}, for every $x \in M$, {where $\mathcal{F}(x)$ denotes the leaf at $x$.}

We say that ${\cF}$ is uniformly expanding if there exists $\lambda > 1$ such that $\forall x \in M$, $\|Df^{-1}_{|T_x(F(x))}\| < \lambda^{-1}$. A measurable partition  $\xi$ is called increasing and subordinated to ${\cF}$ if it satisfies the following properties:
\begin{itemize}
    \item[(a)] $\xi(x) \subseteq {\cF}(x)$ for $\mu$-almost every $x$.
    \item[(b)] $f^{-1}(\xi) \geq \xi$ (increasing property).
    \item[(c)] $\xi(x)$ contains an open neighborhood of $x$ in ${\cF}(x)$ for $\mu$-almost every $x$.
\end{itemize}

Finding a measurable and  increasing partition that aligns with an invariant lamination is typically a complex challenge. However, when dealing with a uniformly expanding foliation that remains invariant under a diffeomorphism, such a partition is always available \cite{Yan21,LY85a,LY85b}. Additionally, for expanding foliations ${\cF}$ it holds
\begin{itemize}
    \item[(d)] $\vee^\infty_{n=0} f^{-n}\xi$ is the partition into points;
    \item[(e)] the largest $\sigma$-algebra contained in $\bigcap^\infty_{n=0} f^n(\xi)$ is ${\cB\cF}$ where ${\cB\cF}$ is the $\sigma$-algebra of ${\cF}$-saturated measurable subsets (union of entire leaves).
\end{itemize}

Given a measure $\mu$ and two measurable partitions $\alpha$, $\beta$ of $M$, the conditional entropy of $\alpha$ given $\beta$ is denoted  by
$$
H_\mu(\alpha|\beta) := \int_M \log(\mu^{\beta}_x(\alpha(x)))d\mu(x),
$$
where $\mu^\beta_x$ is the conditional measure of $\mu$ with respect to $\beta$ as Definition \ref{ConditionalMeasure}.
For any invariant lamination $\mathcal{F}$ with a measurable subordinated increasing partition satisfying all the properties including (d) and (e) above, we define

$$h_\mu(f, {\cF}) := h_\mu(f, \xi) = H_\mu(f^{-1}\xi|\xi) = H_\mu(\xi|f\xi).$$

It can be proved that the above definition is independent of $\xi$ \cite[Lemma 2.8]{HHW17}. 

Let $f$ be  a partially hyperbolic diffeomorphism $f$  such that ${\dime}(E^u_f)\geq 1$ and $\mu\in \cM_f$. We define

$$h_\mu^u(f)=h_\mu(f,{\cF}^u),$$
where ${\cF}^u$ denotes the  unstable foliation given by the unstable manifolds.

Analogously, if the dimension of the stable bundle $E^s_f$ is greater than or equal to 1, we can define the stable entropy for any $\mu \in \cM_f$ as 
$$h^s_\mu(f) := h_\mu(f^{-1},{\cF}^s),$$
where ${\cF}^s$ denotes the stable foliation  given by the stable manifolds.
\subsubsection{ Unstable topological entropy and the variational principle}
Similar to the definition of topological entropy for continuous maps on compact metric spaces, the concept of topological entropy along the unstable foliation can be formulated by considering the notion of spanning (and separated) sets within a collection of compact unstable plaques.

 Let $f$ be  a partially hyperbolic diffeomorphism $f$  such that ${\dime}(E^u)\geq 1$.
We denote by $d^u$ the metric induced by the Riemannian structure on the unstable manifold and let $d_n^u(x, y)=\max _{0 \leq j \leq n-1} d^u\left(f^j(x), f^j(y)\right)$. 
Let $W^u(x, \delta)$ be the open ball inside $W^u(x)$ centred at $x$ of radius $\delta$ with respect to the metric $d^u$. Let $N^u(f, \epsilon, n, x, \delta)$ be the maximal number of points in $\overline{W^u(x, \delta)}$ with pairwise $d_n^u$-distances at least $\varepsilon$. 
We call such a set an $(n, \varepsilon)$ u-separated set of $\overline{W^u(x, \delta)}$. We can alternatively define unstable topological entropy by employing $(n, \epsilon)$ u-spanning sets or open covers, yielding equivalent definitions. 

\begin{defi}\label{def-u-s-entropia}
The unstable topological entropy of $f$ on $M$ is defined by
$$
h^u(f)=\lim_ {\delta \rightarrow 0} \sup _{x \in M} h^u\left(f, \overline{W^u(x, \delta)}\right),
$$
where
$$
h^u\left(f, \overline{W^u(x, \delta)}\right)=\lim _{\epsilon \rightarrow 0} \limsup _{n \rightarrow \infty} \frac{1}{n} \log N^u(f, \epsilon, n, x, \delta) .
$$    
\end{defi}

Analogously, if the dimension of the stable bundle $E^s_f$ is greater than or equal to 1, we can define the stable entropy for any $\mu \in \cM_f$ as $h^s(f) := h^u(f^{-1})$.


Similar to the traditional definition of entropy, we can establish a connection between metric unstable entropy and unstable topological entropy through a variational principle. 
Specifically, according to \cite[Theorem D]{HHW17}, for a $C^{1}$-partially hyperbolic diffeomorphism $f : M \rightarrow M$, the following relation holds:
$$h^u(f) = \sup\{h^u_{\mu}(f) : \mu \in \cM_f\}
\quad \mbox{and} \quad h^u(f) = \sup\{h^u_{\nu}(f) : \nu \in \cM_f^e\}.$$

{An ergodic $f$ invariant measure satisfying $h_\mu(f)=h^u(f)$  is called an ergodic measure of maximal $u$-entropy.}

Another way to define topological unstable entropy  is by considering the unstable volume growth given by Hua, Saghin, and Xia (\cite{HSX08}), reminiscent of the works by Yomdin and Newhouse (\cite{Yom87, New89}). 
In \cite[Theorem C]{HHW17}  it is shown that the unstable topological entropy as defined here coincides with the unstable volume growth.
\subsection{Hausdorff distance}
In this section we will define  the Hausdorff distance, we follow \cite[Section 7.3]{BBI22}.

\begin{defi}
    Let $A$ and $B$  be subsets of the metric space $(M,d)$. The Hausdorff distance between $A$ and $B$, denoted by $d_H(A,B)$ is defined by
    $$d_H(A,B)=\max\{\sup_{a\in A} d(a,B),\sup_{b\in B} d(b,A) \}.$$
    
   \end{defi}
    Let ${\cM}(M)$ be the set of all the closed subsets of $M$. It is known that $({\cM}(M),d_H)$ is a metric space and when $M$ is compact, the metric space $({\cM}(M),d_H)$ is compact. 

    \begin{remark}\label{HausdorffTopology}
Let $(K_n)_{n\in {\bN}}$ be a convergent sequence of compact subsets of the compact metric space $(M,d)$ converging to $K$ in the Hausdorff Metric. If $x_n\in K_n$ forms a convergent sequence in $(M,d)$ converging to $x$, then $x\in K$. Moreover, if $(\mu_k)_{k\in {\bN}}$ are probability measures converging to $\mu$ in the weak * topology, then $\mu(K)\geq \limsup (\mu_n(K_n))$.
    \end{remark}

\section{Consequences of $h^u(f)-h^s(f)>0$}\label{s-consequences-hu>hs} 

Let $\Diff^{1 +}(M)$ be the set of $C^{1 +}$ diffeomorphisms defined in $M$.
From now on, $f$ will be a $C^{1 +}$ partially hyperbolic diffeomorphism with $1$-dimensional central bundle and  non trivial stable and unstable bundles, that is, \linebreak ${\dime}(E^{\sigma})>0 (\sigma=s,u)$.

For $f\in \Diff^{1+}(M)$ and $\mu$ ergodic
 by the   Ledrappier and Young formula, \cite{LY85a, LY85b} and \cite{Bro22}, we get that

\begin{equation}\label{L-Yformula}
    h_\mu(f)\leq h_\mu^u+ \sum_{\lambda_c(f,\mu)>0}\lambda_c(f,\mu) 
\end{equation}
and if the central Lyapunov exponent of $\mu$ is non-positive then, by the definition of unstable metric entropy we get
\begin{equation}\label{u-entropy-metric=entropy-metric}
    h_\mu(f)= h_\mu^u(f). 
\end{equation}

\begin{lemma}\label{mbe-implies-center-lyapunov-exponent<0}
Consider $f:M\to M$ such that  $h^u(f)>h^s(f)$. If $\mu$ is an ergodic measure such that $h_\mu(f)>h^s(f)$, then {$\lambda_c(f,\mu)\leq h^s(f)-h_\mu(f)<0$.}
\end{lemma}

\begin{proof}
Let $\mu$ be an ergodic measure for $f$ with $h_\mu(f)>h^s(f)$.

\begin{claim}\label{Claim1}
     The central Lyapunov exponent of $\mu$ is negative, i.e., $\lambda_c(f,\mu)<0$.
\end{claim}

\begin{proof}
    Suppose that the central Lyapunov exponent of $\mu$ is non-negative, i.e., $\lambda^c(\mu,f)\geq 0$. Applying \eqref{u-entropy-metric=entropy-metric} to the stable metric entropy and considering $f^{-1}$, we obtain
    $$h_\mu(f)=h^s_\mu(f).$$
 Since by hypothesis $h_\mu(f)>h^s(f)$, we get $h_\mu^s(f)>h^s(f)$, which contradicts the variational principle  for stable entropy.
\end{proof}

By $h_\mu(f)>h^s(f)$ and Claim \ref{Claim1}, we have that $\lambda_c(f,\mu)<0$, and therefore, by \eqref{L-Yformula},

$$h_\mu(f)\leq h^s(f)-\lambda_c(f,\mu).$$

Hence, the proof is concluded.
    
\end{proof}

\begin{remark}
    It is worth noting that under the hypotheses of Lemma \ref{mbe-implies-center-lyapunov-exponent<0}, every ergodic measure {with metric entropy bigger  than the stable entropy} is hyperbolic and the metric entropy  coincides with the unstable metric entropy. Furthermore, the central Lyapunov exponent is negative, so its $s$-index is $\operatorname{dim}(E^s)+1$. This leads us to inquire whether two measures with large entropy exhibit homoclinic relation as defined in the definition \ref{def-HR-measures}.
\end{remark}

Now, we will introduce an analogous result for the pressure, but before doing so, let us establish some notations. For every $\mu \in \cM_f$ and $\phi: M \to {\bR}$, a continuous potential, we define:

$$P_\mu(f,\phi) := h_\mu(f) + \int \phi d\mu \quad \mbox{and} \quad 
P_\mu^{u/s}(f,\phi) := h_\mu^{u/s}(f) + \int \phi d\mu.$$

Furthermore,

$$P^{u/s}(f,\phi) := \sup_{\mu \in \cM_f}\left\{P_\mu^{u/s}(f,\phi)\right\}.$$

 Observe that the condition $P^u(f, \phi) > P^s(f, \phi)$ is not sufficient to guarantee positive entropy for the equilibrium states of $(f, \phi)$. When $h^u(f) > \sup \phi > \inf \phi > h^s(f)$, every measure of maximal entropy $\mu$ satisfies \linebreak   $P_\mu(f, \phi) > h^s(f) + \sup \phi$. So, we can state the following

\begin{clly}\label{mbp-implies-center-lyapunov-exponent<0}
    
Let  $\phi:M\to {\bR}$ be a continuous potential  such that $h^u(f)>\sup \phi>\inf\phi>h^s(f)$. If $\mu$ is an ergodic measure such that $P_\mu(f,\phi)>h^s(f)+\sup \phi$ then {$\lambda_c(f,\mu)\leq h^s(f)-h_\mu (f)<0$.}
\end{clly}

\begin{proof}
Let us consider $\mu \in \cM_f^e$ such that $P_\mu(f,\phi) > h^s(f)+\sup \phi$.

Since $\mu$ is a probability measure, we have $\sup \phi \geq \int \phi d\mu$. Therefore, given that $h_\mu(f) + \int \phi d\mu > h^s(f) + \sup \phi$, we can conclude that:
$$h_\mu(f) > h^s(f).$$

Now, since $h^u(f)>h^s(f)$, by Lemma \ref{mbe-implies-center-lyapunov-exponent<0}, we obtain:
$$\lambda_c(f,\phi) \leq h^s(f) - h_\mu(f) < 0.$$

Thus, we have completed the proof.
\end{proof}

Consider the map  $h: {\Diff}^{1+}(M) \to {\bR}$ such that $h(f)$ is the topological entropy of $f$.  The next Lemma shows that the property $h^u(f)>h^s(f)$ is  open: 

\begin{lemma}\label{h^u(g)>h^s(g)}
If $h^u(f) > a_1 > a_2 > h^s(f)$, then there exists a $C^1$ neighborhood ${\cU}$ of $f$ such that for every $g \in {\cU} \cap {\Diff}^{1+}(M)$ the following holds:

\begin{enumerate}
    \item  $h^u(g) > a_1 > a_2 > h^s(g)$,
    \item the function $h:{\Diff}^{1+}(M)\cap \mathcal{U} \to \mathbb{R}$ is continuous,
    \item if $\phi:M\to \bR$ is a H\"older continuous potential with $a_1=\sup \phi$ and $a_2=\inf \phi$,  the function $P(.,\phi):\Diff^{1+}(M)\cap \mathcal{U} \to \mathbb{R}$ is continuous.

\end{enumerate}  
\end{lemma}

\begin{proof}
    The lemma is a consequence of  \cite[Corollaries 3.6 and 3.7]{MP23}. 
\end{proof}

\section{Invariant manifolds for measures with large entropy}\label{Hyperbolic-time}
In this section, we will investigate the {measure} of the set of points with invariant manifolds of uniform size from the perspective of measures with high entropy.

\subsection{Invariant Manifolds with Uniform Size}

{We verified that ergodic measures with large entropy are hyperbolic measures, see Lemma \ref{mbe-implies-center-lyapunov-exponent<0}. Almost every point of a hyperbolic measure possesses stable and unstable invariant manifolds, see Section 
\ref{subsectionHyperbolicMeasures}. However, a crucial point to note is the absence of uniform control over the size of these invariant manifolds. Consequently, there may exist points that are arbitrarily close to each other such that their corresponding unstable and stable manifolds do not intersect. This suggests the potential existence of an infinite number of maximal entropy measures that are not homoclinically related. 
{To avoid this kind of problem, in this section, we construct, for  partially hyperbolic diffeomorphisms with central dimension one a set where the Pesin invariant manifolds have uniform size.}}  

Recall that  $f:M\to M$ denotes a $C^{1 +}$ partially hyperbolic diffeomorphism with $1$-dimensional central direction with non trivial stable and unstable bundles, that is, 
${\dime}(E^{\sigma})>0 (\sigma=s,u)$. Given $1>\rho\geq \xi>0$ where $\xi$ is as in Definition \ref{def-partialh}, define

\begin{equation}\label{Lambarho-definition}
\Lambda_\rho(f)=\Lambda_\rho:=\{x\in M: \forall n\geq 1 \mbox{ it holds } \|Df^n(x)_{|E^c}\|\leq \rho^n \}.    
\end{equation}

\begin{remark}\label{manifoldsexistence}
{Note that $\Lambda_\rho$ is a compact set.  Moreover $f$ is a partially hyperbolic diffeomorphism with central bundle one and such that for every $x\in \Lambda_\rho$ we have that }

$$\|Df^n_x(v)\|< \rho^n \mbox{ for every } v\in E^{c} \mbox{ and } n\geq 1$$ 
{and the splitting $E^{cs}(x)\oplus E^{u}(x)$ varies continuously. Therefore, the classical result presented in \cite{HPS70,MV20} guarantees the existence of invariant  manifolds of uniform size. More precisely, there exists $\delta_r=\delta_r(\rho)$ and $\delta_l=\delta_l(\rho)$ such that for every $x\in \Lambda_{\rho}(f)$  unstable and stable manifold have a size larger than $\delta_r$ and if $x,y\in \Lambda_{\rho}(f)$ and $d(x,y)<\delta_l$ then $W^{u/s}_{\delta_r}(x)\pitchfork W^{s/u}_{\delta_r}(y)$ is non empty. Furthermore, the invariant manifolds exhibit continuous variations concerning both the diffeomorphisms involved and the chosen points. In other words, given $f$ and $g$ are $C^{1}$ close enough in the $C^1$ topology, for any $x$ in $\Lambda_\rho(f)$ and $y$ in $\Lambda_\rho(g)$ where $x$ and $y$ are close, the invariant manifolds $W^{s/u}_{\delta_r}(x)$ are also close to $W^{s/u}_{\delta_r}(y)$ within the $C^1$ topology.}
\end{remark}


\subsection{Measuring $\Lambda_{\rho}$}
Now, we know that points in $\Lambda_{\rho} (f)$ have invariant manifolds of uniform size. We will use this set to prove that there exists a finite number of homoclinic classes containing measures with large entropy. 
Therefore, we aim to demonstrate that the set $\Lambda_\rho(f)$ possesses a positive measure for measures with significant entropy. To achieve our goals, we require the following version of the Pliss Lemma \cite{Pli72}.

\begin{lemma}\cite[Lemma 3.1]{CP18}\label{Pliss lemma}
    For any $\varepsilon>0$, $\alpha_1<\alpha_2$ and any sequence 
    $(a_i)\in (\alpha_1,+\infty)^{{\bN}}$ satisfying 

    $$\limsup_{n\to +\infty}{\frac{a_0+\cdots a_{n-1}}{n}}\leq \alpha_2,$$
    there exists a sequence of integers $0<n_1<n_2<\cdots $ such that

    \begin{itemize}
        \item[(a)] for any $l\geq 1$ and $n>n_l$, one has $\frac{a_{n_l}+\cdots +a_{n-1}}{n -n_l}\leq \alpha_2+\varepsilon$,
        \item[(b)] the upper density $\limsup \frac{l}{n_l}$ is larger than $\frac{\varepsilon}{\alpha_2+\varepsilon-\alpha_1}$.
    \end{itemize}
\end{lemma}
 We define $m(f):=\min_{x\in M}\{\log\|Df(x)_{|E^c_x}\|\}$. 
 Note that $m(f)$ {is negative because there are points with central Lyapunov exponent negative} and it  varies continuously with $f$ on the $C^1$ topology.

\begin{lemma}\label{Hyerbolictime}
    Suppose that $h^u(f)>h^s(f)$. If $\mu\in \cM_f^e$ with $h_\mu(f)>h^s(f)$ then for every $e^{\lambda_c(f,\mu)}<\rho<1$ we have that $\mu(\Lambda_\rho)\geq\frac{\log(\rho)-\lambda_c(f,\mu)}{\log (\rho)-m(f)}.$
\end{lemma}

\begin{proof}
    Let $\mu \in \cM_f$ be an ergodic measure with $h_\mu(f)>h^s(f)$. Since $h^u(f)>h^s(f)$ and $h_\mu(f)> h^s(f)$, Lemma \ref{mbe-implies-center-lyapunov-exponent<0} implies that $\lambda_c(f,\mu)<h^s(f)-h_\mu(f)<0$. Now, let $0<e^{\lambda_c(f,\mu)}<\rho<1$ and define the set
    \[
    \Lambda(f)=\Lambda_\rho\cap \mathcal{R} \cap \Lambda_\mu,
    \]
    where $\mathcal{R}$ is the Lyapunov regular set and $\Lambda_\mu$ are the points in $M$ with central Lyapunov exponent equal to $\lambda_c(f,\mu)$. Note that $\mu(\Lambda(f))=\mu(\Lambda_\rho)$.

    For every $x\in \Lambda(f)$, we have that
    \[
\begin{aligned}
\lim_{n\to \infty}\frac{1}{n}\log(\|Df^n(x)_{|E^c_x}\|) &= \lim_{n\to \infty}\frac{1}{n}\sum_{j=0}^{n-1}\log(\|Df(f^j(x))_{|E^c(f^j(x))}\|) \\
&= \lambda_c(f,\mu).
\end{aligned}
\]

    We will now apply Lemma \ref{Pliss lemma}. Take $\varepsilon=\log(\rho)-\lambda_c(f,\mu)$, $\alpha_1=m(f)$, $\alpha_2=\lambda_c(f,\mu)$, and consider the sequence 
    \[
    (\log(\|Df(f^j(x))_{|E^c(f^j(x))}\|))_{j\in {\bN}}.
    \]

    This results in a sequence of integers $(n_l)_{l\in {\bN}}$ such that for each $l\in {\bN}$ and $n>n_l$, we find that
    
    $$
    \frac{1}{n-n_l}\sum_{j=n_l}^{n-1}\log\|Df({f^j(x))_{|E^c_{f^j(x)}}}\|<\log(\rho).
    $$

    Consequently, we obtain $\|Df^n(f^{n_l}(x))_{|E^c(f^{n_l}(x))}\|<\rho^n$ for every $n\in {\bN}$. Thus, for each $l$, $f^{n_l}(x)\in \Lambda(f)$. Applying Birkhoff's theorem and the Pliss Lemma, we arrive at the following estimate:
    \begin{equation}
        \mu(\Lambda_\rho)\geq \limsup_{l\to \infty} \frac{l}{n_l}\geq\frac{\log(\rho)-\lambda_c(f,\mu)}{\log (\rho)-m(f)}>0.
    \end{equation}
\end{proof}

\begin{clly}\label{HyperbolictimePressure}
     Let $\phi:M\to {\bR}$ be a continuous potential  such that $h^u(f)>\sup \phi>\inf\phi>h^s(f)$. 
     If $\mu$ is an ergodic measure such that $h_\mu(f)+\int \phi d\mu>h^s(f)+\sup \phi$  then for every $e^{\lambda_c(f,\mu)}<\rho<1$ we have that 
     $$\mu(\Lambda_\rho)\geq\frac{\log(\rho)-\lambda_c(f,\mu)}{\log (\rho)-m(f)}.$$
\end{clly}

\begin{proof}
   As in the proof of Corollary \ref{mbp-implies-center-lyapunov-exponent<0}, the inequality $h_\mu(f)+\int \phi d\mu>h^s(f)+\sup \phi$ implies 
$$h_\mu(f)>h^s(f)$$
so, by Lemma \ref{Hyerbolictime} we have  that  
$$\mu(\Lambda_\rho)\geq\frac{\log(\rho)-\lambda_c(f,\mu)}{\log (\rho)-m(f)}.$$
\end{proof}

\begin{clly}\label{Hyperbolictime-uniform-measure}
   Suppose that  $h^u(f)>h^s(f)$. For every $h^u(f)>a>h^s(f)$ and $\frac{h^s(f)-a}{2}<\log(\rho)<0$ there exist a $C^1$ neighborhood $\cV$ of $f$ and a real number $\beta>0$ such that for every $g\in \cV \cap {\Diff}^{1+}(M)$ holds.
    \begin{enumerate}
        \item[(a)]\label{4.5a}  $h^u(g)>\frac{h^u(f)+a}{2}$ and  $h^s(g)<\frac{h^s(f)+a}{2}$,
        \item[(b)] If $\mu\in \cM_{g}^e$ with $h_\mu(g)\geq a$, then  $\mu(\Lambda_\rho(g))\geq \beta$.
    \end{enumerate}   
\end{clly}
\begin{proof}
For the item (a), it is enough to consider $a_1=\frac{h^u(f)+a}{2}$ and $a_2=\frac{h^s(f)+a}{2}$ in Lemma \ref{h^u(g)>h^s(g)}.

For the item (b), let $\mu\in \cM_{g}^e$ such that $h_\mu(g)\geq a$. By Lemma \ref{mbe-implies-center-lyapunov-exponent<0}, we have:
\begin{equation}\label{Cor4.4eq1}
\lambda_c(g,\mu) \leq h^s(g)-h_\mu(g).  
\end{equation}
Using the first item, we also have:
\begin{equation}\label{Cor4.4eq2}
h^s(g)-h_\mu(g)<\frac{h^s(f)+a}{2}-h_\mu(g).
\end{equation}
Since  \eqref{Cor4.4eq1}, \eqref{Cor4.4eq2}  and $h_\mu(g)\geq a$, it follows:
\begin{equation*}
\lambda_c(g,\mu)<\frac{h^s(f)-a}{2}.
\end{equation*}
Consider $\log(\rho)>\frac{h^s(f)-a}{2}$, so, $e^{\lambda_c(g,\mu)}<\rho<0$. 
Therefore, by Lemma \ref{Hyerbolictime}, we have:
$$\mu(\Lambda_\rho(g))\geq \frac{\log(\rho)-\lambda_c(g,\mu)}{\log (\rho)-m(g)}\geq \frac{\log(\rho)-\frac{h^s(f)-a}{2}}{\log(\rho)-m(g)}>0.$$
Note that since $m(g)$ varies continuously in the $C^1$ topology, we can choose $\beta>0$ such that:
$$\frac{\log(\rho)-\frac{h^s(f)-a}{2}}{\log (\rho)-m(g)}>\beta$$
for every $g\in \cV\cap {\Diff}^{1+}(M)$, where ${\cV}$ is an open neighborhood of $f$ as in item (a) and with small size such that $m(g)$ is close to $m(f)$.  Thus, we conclude the proof.
\end{proof}

\begin{clly}\label{measures-uniform-pressure}
   Let $\phi:M\to \mathbb{R}$ be a continuous potential such that $h^u(f)>\sup \phi>\inf\phi>h^s(f)$. 
   For every $a$ such that  $h^u(f)+\inf \phi>a>h^s(f)+\sup \phi$  and $\frac{h^s(f)-a+\inf \phi}{2}<\log(\rho)<1$
  it holds
   \begin{enumerate} 
        \item [(a)] For every $g\in \cV$, 
        $$h^u(g)+\inf \phi>\frac{h^u(f)+\inf \phi+a}{2} \quad \mbox{and}\quad h^s(g)+\sup \phi<\frac{h^s(f)+\sup \phi+a}{2},$$
        \item[(b)] If $\mu\in \cM_{g}^e$ with $P_\mu(g)\geq a$, then  $\mu(\Lambda_\rho(g))\geq \beta$.
    \end{enumerate}   
\end{clly}
\begin{proof}
For  item (a), it is enough to consider $a_1=\frac{h^u(f)+\inf \phi+a}{2}$ and $a_2=\frac{h^s(f)+\sup \phi+a}{2}$ in Lemma \ref{h^u(g)>h^s(g)}.

For item (b), let $g\in \cV\cap {\Diff}^{1+}(M)$ and $\mu \in\cM_g$ such that $P_\mu(f)\geq a$, where ${\cV}$ is the neighborhood of $f$ given by item $(a)$. 
By Corollary \ref{mbp-implies-center-lyapunov-exponent<0}, we have:
\begin{equation}\label{Coro46eq1}
\lambda_c(g,\mu)\leq h^u(g)-h_\mu(g).
\end{equation}
Now, on one side, since $g\in \cV\cap {\Diff}^{1+}(M)$, we have $h^s(g)+\sup \phi<\frac{h^s(f)+a+\sup \phi}{2}$ and therefore:
\begin{equation}\label{Coro46eq2}
h^s(g)<\frac{h^s(f)+a-\sup \phi}{2}.
\end{equation}
On the other side, as $h_\mu(f)+\int \phi d\mu\geq a$ and $\sup \phi\geq \int \phi d \mu $, we have:
\begin{equation}\label{Coro46eq3}
h_\mu(f)\geq a-\sup \phi.
\end{equation}
Applying \eqref{Coro46eq2} and \eqref{Coro46eq3} in \eqref{Coro46eq1} we get:
\begin{equation*}
\lambda_c(g,\mu) < \frac{h^s(f)+\sup \phi-a}{2}.
\end{equation*}
Now, consider $\log(\rho)>\frac{h^s(f)+\sup \phi-a}{2}$. So, $e^{\lambda_c(g,\mu)}<\rho<1$, and therefore, by Corollary \ref{HyperbolictimePressure}:
$$\mu(\Lambda_\rho(g))\geq \frac{\log(\rho)-\lambda_c(g,\mu)}{\log (\rho)-m(g)}\geq \frac{\log(\rho)-\frac{h^s(f)+\sup \phi-a}{2}}{\log (\rho)-m(g)}.$$
Note that since $m(g)$ varies continuously in the $C^1$ topology, we can choose $\beta>0$ such that:
$$ \frac{\log(\rho)-\frac{h^s(f)+\sup \phi-a}{2}}{\log (\rho)-m(g)}>\beta$$
for every $g\in \cV\cap {\Diff}^{1+}(M)$, reducing the size of the neighborhood $\cV$ when necessary. Thus, we conclude the proof.
\end{proof}
\section{Proof of Theorems \ref{MainTheoA} and \ref{MainTheoB} }\label{Proof-A-B} 
In this section, we prove theorems \ref{MainTheoA} and \ref{MainTheoB}. 
Recall that $f:M\to M$ denotes a $C^{1+}$ partially hyperbolic diffeomorphism of a compact manifold $M$ with $TM=E^u \oplus E^c \oplus E^s $ and ${\dime}(E^c)=1$.
\begin{proof}[Proof of  Theorem A]
 Let $a>0$ be a real number such that $h^u(f)\geq a> h^s(f)$. By Lemma \ref{mbe-implies-center-lyapunov-exponent<0}, we have that for every ergodic measure $\mu$ with $h_\mu(f)\geq a$, it holds:
$$\lambda_c(f,\,u)\leq h^s-a<0.$$

Fix $e^{h^s(f)-a}<\rho <1$ and consider $\Lambda_\rho\subset M$, and $\delta_l>0$ as in Definition \ref{def-partialh} and Remark \ref{manifoldsexistence} respectively.

\begin{claim}\label{NumberOfHC}
    There exist $K_a\in \mathbb{N}$ and finitely many ergodic measures $\mu_1,\cdots \mu_{K_a}$  with $h_{\mu_i}(f)\geq a$ such that if $\mu\in\cM_f^e$ with $h_\mu(f)\geq a$, then $\mu\in HC(\mu_i)$ for some $i\leq K_a$.
\end{claim}

    \begin{proof}
By Lemma \ref{Hyerbolictime}, for every ergodic $\mu\in \cM_f^e$ with $h_\mu(f)\geq a$, we have:

$$\mu(\Lambda_\rho)> 0.$$

Thus, for every ergodic measure $\mu$ of $f$ with $h_\mu(f)\geq a$, there exists $x_{\mu}\in \Lambda_\rho$ such that:

\begin{equation}\label{eq5.1.1}
    \mu(B_{\frac{\delta_l}{3}}(x_{\mu})\cap \Lambda_\rho)>0.
\end{equation}

Now, let $(\mu_j)_{j\in \bN}$ be a sequence of distinct  ergodic measures of $f$ such that $h_{\mu_j}(f)\geq a$ for every ${j\in \bN}$ and choose $x_{\mu_j}$ such that  \eqref{eq5.1.1} holds. 
Due to the compactness of $M$, there are $n,\, m\,\in\bN, n\neq m,$ such that:

$$d(x_{\mu_{n}},x_{\mu_{m}})<\frac{\delta_l}{3}.$$

By Remark \ref{manifoldsexistence}, for $\mu_{n}$-almost every  $x_1\in B_{\frac{\delta_l}{3}}(x_{\mu_{n}})\cap \Lambda_\rho$ and $\mu_{m}$-almost every $x_2\in B_{\frac{\delta_l}{3}}(x_{\mu_{m}})\cap \Lambda_\rho$, we obtain that $W^{u/s}(x_1)\pitchfork W^{s/u}(x_2)$ are both not empty, and therefore, $\mu_{n}\sim \mu_{m}$.

Therefore, there is only a finite number of homoclinic classes in the set of hyperbolic measures containing measures with metric entropy greater than or equal to $a$. Finally, since  every ergodic measure with metric entropy  greater or equal than $a$ respect $f$ is a hyperbolic measure, we can choose $\mu_1, \ldots, \mu_{K_a}$ as representatives of each distinct homoclinic class, each having entropy greater than or equal to $a$. Thus, we conclude the proof of the claim.
\end{proof}
     
To finish the proof, we simply choose $p_1, \ldots, p_{K_a}\in M$ with hyperbolic periodic orbits ${\cO}_1, \ldots, {\cO}_{K_a}$ such that ${\cO}_i\sim \mu_i$ for every $1\leq i\leq K_a$ and $HC(\cO_i)=HC_T(\mu_i)$. This is possible by Proposition \ref{HoclinicClassRelation}.    
    \end{proof}

    \begin{clly}\label{finite-MME-forf}
            If $h^u(f)>h^s(f)$ then $f$ has only a finite number of measures of maximal entropy.
    \end{clly}

\begin{proof}
{
 Since $f$ is a partially hyperbolic system with a central bundle of dimension one, it possesses measures of maximal entropy. By Lemma \ref{mbe-implies-center-lyapunov-exponent<0}, every ergodic measure of maximal entropy is hyperbolic, so according to Proposition \ref{HoclinicClassRelation}, for every ergodic m.m.e.\ $\mu$  we can select an orbit $\mathcal{O}\in \operatorname{Per}_h(f)$ such that $\mu \sim {\cO}$ and by Proposition \ref{criterio-MME} there is at most one hyperbolic measure of maximal entropy associated with ${\cO}$. 
   Consequently, $\mu$ is the unique measure of maximal entropy homoclinically related  to ${\cO}$. Therefore every homoclinic class in $\cM_f^h$ has at most one measure of maximal entropy}. 
   
    Finally, by  Theorem \ref{MainTheoA}, the number of homoclinic classes in $\cM_f^h$ containing a measure of maximal entropy is finite. Therefore $f$ has finitely  many ergodic measures of maximal entropy.
   
    \end{proof}

    \begin{clly}\label{Finite-MME}
       If $h^u(f)>h^s(f)$, then there exists a $C^1$ neighborhood $\cU$ of $f$ such that  every $C^{1+}$ diffeomorphism $g\in \cU$ has only a finite number of measures of maximum entropy.

    \end{clly}
    
    \begin{proof} By Lemma $\ref{h^u(g)>h^s(g)}$  there exists a neighborhood $\cU$ of $f$ such that if $g\in \cU\cap \operatorname{Diff}^{1+}(M)$ so $h^u(g)>\frac{h^u(f)+h^s(f)}{2}>h^s(g)$. Finally, by Corollary \ref{finite-MME-forf}  every $g\in \cU \cap {\Diff}^{1+}(M)$  has finitely many ergodic measures of maximal entropy.
        
    \end{proof}

\begin{proof}[Proof of Theorem \ref{MainTheoB}]\label{ProofTheoB}
Let $\cU$ be a $C^1$ neighborhood of $f$ as in Corollary \ref{Finite-MME} and Lemma \ref{h^u(g)>h^s(g)}, which means that for every $g\in \cU \cap {\Diff}^{1+}(M)$, $h^u(g)>h^s(g)$, $g$ has many finite measures of maximal entropy, and the entropy varies continuously in $\cU \cap {\Diff}^{1+}(M)$.

Let $k>0$  be the number of measures of maximal entropy of $f$ given by Corollary  \ref{finite-MME-forf}. Let $\nu_1, \ldots, \nu_k$ be these ergodic measures of maximal entropy of $f$, and $p_1(f), \ldots, p_{k}(f)$ be periodic points of $f$ with orbit ${\cO}_1, \ldots, {\cO}_k$, given by Proposition \ref{HoclinicClassRelation} such that ${\cO}_i \sim \nu_i$, $HC({\cO}_i)=HC_T(\nu_i)$ and 
$$h(f_{|HC({\cO}_i)})=h(f).$$

    Let us consider a sequence of $C^{1+}$ diffeomorphisms, denoted as $g_n$, where each $g_n$ belongs to $\cU$ and the sequence $(g_n)_n$ converges to $f$ in the $C^1$ topology. Let $p_i(g_n)$ be the unique hyperbolic periodic point of $g_n$ derived from the structural stability of $p_i(f)$. 
    
    \begin{lemma}\label{p(g)SimMuj}
 If $\mu_n \in \cM_{g_n}^e$ is a m.m.e.\ for $g_n$, then there exists an index $j$ and a subsequence $(\mu_{n_l})_{l\geq 0}$ such that $\mu_{n_l}\sim p_j(g_{n_l})$ for every $l \geq 0$.
            \end{lemma}
        \begin{proof}
    Let $\rho$ be a real number such that  $h^s(g_n)-h^u(g_n)<\log(\rho)<0$ for every $n$ and let $\Lambda_\rho(g_n)$ be as defined in \eqref{Lambarho-definition}. 
    By Corollary \ref{Hyperbolictime-uniform-measure}, there exists $\beta>0$ such that $\mu_{n}(\Lambda_{\rho}(g_n))\geq 2\beta$ and  therefore we can choose a compact subset $\tilde{\Lambda}_\rho(g_n)\subset \Lambda_\rho(g_n)$ of regular points of $\mu_n$ with $\mu_n(\tilde{\Lambda}_\rho(g_n))\geq \beta$.  
    Let $n_l$ be an increased sequence of natural numbers such that $\mu_{n_l}$ converges to $\nu$ in the weak * topology and $\Lambda_\rho(g_{n_l})$ converges to a compact set $K$ in the Hausdorff metric. 
    Observe that $K\subset \Lambda_\rho(f)$; so, by Remark \ref{HausdorffTopology}, there is a compact subset $\tilde{\Lambda}_\rho(f)\subset K$ of regular points for $f$ with $\nu(\Tilde{\Lambda}_\rho(f))\geq \frac{\beta}{2}$.

Observe that $\nu$ is not necessarily an ergodic measure  however it  is a maximal entropy measure of $f$. In fact, since $\mu_n\in \cM_{g_n}^e$ and $\mu_n$ is m.m.e.\ for $g_n$, by  Lemma \ref{mbe-implies-center-lyapunov-exponent<0}, we have that \eqref{u-entropy-metric=entropy-metric} holds, that is, $$h^u_{\mu_n}(g_n) = h_{\mu_n}(g_n).$$ 
            
   Furthermore, as shown in \cite[Theorem A]{Yan21}, we have:

\[
\limsup_n h^u_{\mu_{n_l}}(g_n) \leq h_\nu^u(f).
\]

Since each $\mu_n$ is  a measure of maximal entropy for $g_n$, the topological entropy is a continuous function on ${\cU} \cap {\Diff}^{r+1}(M)$ and $h^u_\mu(f)\leq h_\mu(f)$, it follows that $\nu$ is also a measure of maximal entropy for $f$, therefore, the measure $\nu$ is a convex combination of $\nu_1,\cdots,\nu_k$, that is, $\nu=\sum_ia_i\nu_i$ where $0\leq a_i\leq 1$ so, there is $j$ such that $\nu_j(\Tilde{\Lambda}_\rho(f))\geq \frac{\beta}{2}$.

   By Definition \ref{Characterization-Homoclinic-class} we can choose $x\in  W^s(p_j(f)) \cap W^u(p_j(f))$ such that 
   
   $$\nu_j(B_{\frac{\delta_1}{4}}(x)\cap \Tilde{\Lambda}_\rho(f))>0,$$
   where $0<\delta_1$ is small enough such that for every $z\in B_{\frac{\delta_1}{2}}(x)\cap \Tilde{\Lambda}_\rho(f)$,
   
   $$W^{s/u}_{\delta_1}(z)\pitchfork W^{s,u}_{\delta_1}(x)\neq \emptyset.$$
   
   Since the invariant manifolds of the points $z\in\Lambda_\rho(g_{n_l})$ vary continuously in $g_{n_l}$ and $z$, for every ${n_l}$ large enough, there exists  $y_{n_l}\in \tilde{\Lambda}_\rho(g_n)$ such that $d(y_{n_l},x)<\delta/4$ and such that
   $$\mu_{{n_l}}(B_{\frac{\delta_1}{4}}(y_{n_l})\cap \Tilde{\Lambda}_\rho(g_{n_l}))>0,$$
    and for every $z\in B_{\frac{\delta_1}{4}}(y_{n_l})\cap \Tilde{\Lambda}_\rho(g_{n_l})$ we have

    $$W^{s/u}_{\delta_1}(z)\pitchfork W^{s,u}_{\delta_1}(x)\neq \emptyset.$$

  Since the invariant manifolds vary continuously with  $g_n$, the invariant manifolds $W^{s/u}(p_j(g_{n_l}))$ converge to $W^{s/u}(p_j(f))$ in the $C^1$ topology, and so  there exists a sequence $x_{n_l}\in W^s(p_j(g_{n_l}))\cap W^u(p_j(g_n))$ such that $(x_{n_l})_{l\in{\bN}}$ converges to $x$ and $W^{s/u}_{\delta_1}(x_{n_l})$ converges to $W^{s/u}_{\delta_1}(x)$ in the $C^1$ topology. Therefore, for $n$ large enough,
   
    $$W^{s/u}_{\delta_1}(z)\pitchfork W^{s,u}_{\delta_1}(x_{n_l})\neq \emptyset$$
    for every $z\in B_{\frac{\delta_1}{2}}(y_{n_l})\cap \Tilde{\Lambda}_\rho(g_{n_l})$. Thus, we have that $p_j(g_{n_l})\sim \mu_{{n_l}}$, for every ${n_l}$ big and the proof is concluded.

    \end{proof}

Now, observe that $g_n$ cannot have two ergodic measures of maximal entropy that are homonically related to the same $p_i(g_n)$. 
Thus, the proof is concluded.
\end{proof}

  Let $\mu_1, \cdots , \mu_k$ be the ergodic measures of maximal entropy of $f$. Also, let $p_i(f)$ denote the hyperbolic point associated to $\mu_i$ for $i=1,\cdots k$. For a $C^1$ diffeomorphism $g$ sufficiently close to $f$, consider the hyperbolic periodic points $p_i(g)$  
such that each $p_i(g)$  is near $p_i(f), \, i=1,\cdots k$. We obtain the following corollary:

\begin{clly}
{Let $f:M\to M$ be a $C^{1+}$ partially hyperbolic diffeomorphism of a closed manifold $M$ with $TM=E^u \oplus E^c \oplus E^s $ and $\operatorname{dim}(E^c)=1$. If $h^u(f)>h^s(f)$, then there exists a $C^1$ neighborhood $\mathcal{U}$ of $f$ such that for every $C^{1+}$ diffeomorphism $g\in \mathcal{U}$, if $\mu$ is an ergodic measure of maximal entropy for $g$ then $\mu$ is homoclinically related to the orbit of $p_j(g)$ for some $j\in {1,\cdots, k}$.}
\end{clly}

\begin{proof}
{Assume, by contradiction, the existence of a sequence of $C^{1+}$ diffeomorphisms $(g_n)_{n\in\mathbb{N}}$ converging to $f$, and let $\mu_n$ be an ergodic measure of maximal entropy of $g_n$ such that $\mu_n$ is not homoclinically related to any $p_i(g_n)$ for $i=1,\cdots, k$. This sequence contradicts Lemma \ref{p(g)SimMuj}. Thus, we conclude the proof.}
\end{proof}

\section{Proof of theorem \ref{MainCentralLyapunovexponentfarfrom zero}}

In this section, we prove Theorem \ref{MainCentralLyapunovexponentfarfrom zero}.
 Recall that $f:M\to M$ denotes a $C^{1+}$ partially hyperbolic diffeomorphism of a compact manifold $M$ with $TM=E^u \oplus E^c \oplus E^s $ and $\dim(E^c)=1$. Moreover, the sets $\cM^+_f$ and $\cM^-_f$ are defined as follows:
$$\cM^+_f=\{\mu \in \cM_f^e: \lambda_c(f,\mu)\geq 0\} \mbox{ and }\cM^-_f=\{\mu \in \cM_f^e: \lambda_c(f,\mu)\leq 0\} $$
and the values $h^-(f)$ and $h^-(f)$ are given by
$$h^-(f)=\sup_{\mu \in \cM_f^-}\{h_\mu(f)\} \mbox{ and } h^+(f)=\sup_{\mu \in \cM_f^+}\{h_\mu(f)\}.$$

Note that if $h^+(f)\neq h^-(f)$, then either $h^+(f)>h^-(f)$ or $h^+(f)< h^-(f)$. We will consider the case where $h^-(f)> h^+(f)$; the proof for $h^+(f)>h^-(f)$ is analogous. We prove the following proposition.

\begin{prop}\label{farfromzero}
Let $f:M\to M$ be a $C^{1+}$ partially hyperbolic diffeomorphism of a compact manifold $M$ with $TM=E^u \oplus E^c \oplus E^s $ and $\operatorname{dim}(E^c)=1$. Then, $h^-(f)>h^+(f)$ if and only if there exist $\lambda<0$ and $h(f)\geq \tilde{h}>0$ such that if $\mu$ is an ergodic measure of $f$ with $h_\mu(f)>\tilde{h}$, then $\lambda_c(f,\mu)<-\lambda$.
\end{prop}

\begin{proof}
Sufficiency has been proven in  \cite[Equation (9.2)]{CT21}, and  necessity follows directly from the definition of $h^-$ and $h^+$.
\end{proof}

\begin{lemma}\label{HyerbolictimeforC}
    Suppose that $h^-(f)>h^+(f)$. If $\mu\in \cM_f^e$ with $h_\mu(f)>\tilde{h}$ then for every $e^{-\lambda}<\rho<1$ we have that $\mu(\Lambda_\rho)\geq\frac{\log(\rho)-\lambda_c(f,\mu)}{\log (\rho)-m(f)}.$
\end{lemma}
The proof of this lemma is analogous to the proof of Lemma \ref{Hyerbolictime} and so we omit it.

\begin{proof}[Proof of Theorem \ref{MainCentralLyapunovexponentfarfrom zero}] Note that for $a>\tilde{h}$ by Lemma \ref{HyerbolictimeforC} we have the same conclusion than Theorem $\ref{MainTheoA}$. So, we can conclude that $f$ has finitely many ergodic measure of maximal entropy.
    
\end{proof}


\section{Proof of Theorems \ref{MainTheoC} and \ref{MainTheoD}}\label{Proof-C-D}
In this section, we prove Theorems \ref{MainTheoC} and \ref{MainTheoD}. 
As before, $f:M\to M$ denotes a $C^{1+}$ partially hyperbolic diffeomorphism of a compact manifold $M$ with $TM=E^u \oplus E^c \oplus E^s $ and ${\dime}(E^c)=1$ and $\phi:M\to \bR$ denotes a H\"older continuous potential. 

\begin{proof}[Proof of  Theorem \ref{MainTheoC}]
            Suppose that $h^u(f)>\sup \phi>\inf\phi>h^s(f)$. Let $b>0$ be a real number such that $P(f,\phi) \geq b> h^s(f)+\sup \phi$. Note that each ergodic measure $\mu$ of $f$ with $P_\mu(f,\phi)\geq b$ holds that $$h_\mu(f)\geq b-\sup \phi>h^s (f).$$ 

            Consider $a=b-\sup \phi$ in Theorem $\ref{MainTheoA}$ so  there exist $K_a\in \mathbb{N}$, hyperbolic periodic orbits ${\cO}_1\cdots {\cO}_{K_a}$ such that if $\mu\in \cM_f^e$ with $h_\mu(f)\geq b-\sup\phi $,  $\mu \sim {\cO}_i$ for some $1 \leq i\leq K_a$. 

    Finally, we choose the $K_b$ hyperbolic periodic orbits \{$\mathcal{O}_1,\cdots , \mathcal{O}_{K_b}$\} which are homoclinically related to an ergodic measure that has  metric pressure associated greater or equal to $b$. This concludes the proof.
\end{proof}

\begin{clly}\label{finite-EE-forf}
    If $h^u(f)>\sup \phi>\inf \phi> h^s(f)$ then $f$ has only a finite number of ergodic equilibirum states.
\end{clly}

\begin{proof}
    
    As proven in Corollary \ref{finite-MME-forf}. Subsection \ref{subsebtion-dimEc1-mme}, Theorem \ref{criterio-MME}, Proposition \ref{HoclinicClassRelation} and Lemma \ref{mbp-implies-center-lyapunov-exponent<0} imply that  equilibrium states exist for $(f,\phi)$, each ergodic equilibrium state $\mu$ is hyperbolic and there exists an orbit ${\cO}\in {\Per}_h(f)$ such that 
    $\mu$ is the unique equilibrium state homoclinically related to it. Consequently, every homoclinic class in $\cM_f^h$ can have, at most, one ergodic equilibirum states.

    By Theorem \ref{MainTheoC} the number of homoclinic classes in $\cM_f^h$  with equilibrium states is finite. Therefore $(f,\phi)$ has only a finite number of ergodic equilibrium states.  

    \end{proof}

\begin{clly}\label{finiteEE}
If $h^u(f)>\sup \phi>\inf \phi>h^s(f)$, then there exists a $C^1$ neighborhood $\cU$ of $f$ such that for every $C^{1+}$ diffeomorphism $g\in \cU$, $(g,\phi)$ has only a finite number of equilibrium states.
\end{clly}

\begin{proof}
    Taking $a_1= \sup \phi$ and $a_2=\inf \phi$ in Lemma \ref{h^u(g)>h^s(g)} there exists a $C^1$ neighborhood $\cU$ of $f$ such that for every $g\in {\Diff}^{1+}(M)\cap \cU$ it holds $h^u(g)>\sup \phi >\inf \phi>h^s(g)$. 
    So, applying Corollary  \ref{finite-EE-forf} we get that each $g\in {\Diff}^{1+}(M)\cap \cU$  is such that $(g,\phi)$ has only a finite number of ergodic equilibrium states.
\end{proof}

{\em{Proof of  Theorem \ref{MainTheoD}}} Suppose that $h^u(f)>\sup \phi>\inf\phi>h^s(f)$.
    Consider $a_1= \sup \phi$ and $a_2=\inf \phi$, by Lemma \ref{h^u(g)>h^s(g)} and Corollary \ref{finiteEE} there exists a $C^1$ neighborhood $\cU$ of $f$ such that  each $g\in {\Diff}^{1+}(M)\cap \cU$ holds $h^u(f)>\sup \phi >\inf \phi>h^s(f)$ and $(g,\phi)$ has finitely many ergodic equilibrium states.

    By  Corollary \ref{finite-EE-forf} the number of equilibrium estates for $(f,\phi)$ is $k>0$. Let $\nu_1,\cdots \nu_k$   be these ergodic equilibrium states and $p_1(f)\cdots p_k(f)$ the  periodic hyperbolic points with orbits ${\mathcal{O}}_1\cdots {\cO}_k$ respectively of $f$ such that ${\cO}_{i}\sim \mu_i$ and

    $$P(f_{|HC(\mathcal{O}_i)},\phi_{|HC(\mathcal{O}_i)})=P(f,\phi).$$

    Let $(g_n)_{n\in \bN}$ be a sequence of $C^{1+}$ diffeomorphisms in $\cU$ converging to $f$ in the $C^1$ topology. For every $g_n$ and $p_i(f)$ consider $p_i(g_n)$ the hyperbolic periodic point given by the structural stability of $p_i(f)$.

    \begin{lemma}\label{p(g)SimMujP}
 If $\mu_n \in \cM_{g_n}^e$ is an equilibrium state for $g_n$, then there exists an index $j$ and a subsequence $(\mu_{n_l})_{l\geq 0}$ such that $\mu_{n_l}\sim p_j(g_{n_l})$ for every $l \geq 0$.
       
    \end{lemma}

    \begin{proof}                                                                                                                                                          

     We will proceed as in Lemma \ref{p(g)SimMuj}. Let $\rho$  be a real number as in Corollary \ref{measures-uniform-pressure} and for every $n$ consider  $\Lambda_\rho(g_n)$  as in \eqref{Lambarho-definition}. 
     By Corollary \ref{measures-uniform-pressure}, there exists $\beta>0$ such that $\Lambda_\rho(g_n)\geq 2\beta$ and  therefore we can choose a compact subset $\tilde{\Lambda}_\rho(g_n)\subset \Lambda_\rho(g_n)$ of regular points with $\mu(\tilde{\Lambda}_\rho(g_n))\geq \beta$.  
     Let $n_l$ be an increasing sequence of natural numbers such that $\mu_{n_l}$ converges to $\nu$ in the weak * topology and $\Lambda_\rho(g_{n_l})$ converges to a compact set $K$ in the Hausdorff metric. Observe that $K\subset \Lambda_\rho(f)$; so, by Remark \ref{HausdorffTopology}, there is a subset $\tilde{\Lambda}_\rho(f)\subset K$ as a compact set of regular points for $f$ with $\nu(\Tilde{\Lambda}_\rho(f))\geq \frac{\beta}{2}$.

The measure  $\nu$ is an equilibrium state of $(f,\phi)$. In effect, since $\mu_n\in \cM_{g_n}^e$  is an equilibrium state of $(g_n,\phi)$, by Corollary \ref{mbp-implies-center-lyapunov-exponent<0} and Equation \eqref{L-Yformula}, we have that 
$$h^u_{\mu_n}(g_n) = h_{\mu_n}(g_n).$$ 
            
   Furthermore, as shown in \cite[Theorem A]{Yan21}, we have:

\[
\limsup_{n_l} h^u_{\mu_{n_l}}(g_n) \leq h_\nu^u(f)
\]
and as $\mu_{n_l}\to \nu$ when $n\to \infty$,

$$\lim_{n\to \infty}\int \phi d\mu_{n_l}=\int \phi d\nu.$$

Therefore, since each $\mu_n$ is  an equilibrium state  for $(g_n,\phi)$, the topological pressure is a continuous function on ${\cU} \cap {\Diff}^{r+1}(M)$ ,  and $h^u_\mu(f)\leq h_\mu(f)$, we conclude that $\nu$ is also an equilibrium state of $(f,\phi)$.

    Thus, the measure $\nu$ is a convex combination of $\nu_1,\cdots,\nu_k$, and there is $j$ such that $\nu_j(\Tilde{\Lambda}_\rho(f))\geq \frac{\beta}{2}$. Consider $p_j(f)$ such that $p_j\sim \mu_j$ and repeating the argument in \ref{p(g)SimMujP} Theorem \ref{MainTheoB} ,  by the continuity of the invariant manifolds with $g_n$, we conclude that for every $l$ large enough $p_i(g_{n_l})\sim \mu_{n_l}$. Thus, the proof is concluded.
    
            Finally, with the same argument of the proof of  Theorem $\ref{MainTheoB}$, we conclude that $g_n$ has at most $k$ ergodic equilibrium states when $n$ is big enough.
            
\end{proof}

{\em{Acknowledgements.}} We are thankful to Jiagang Yang and Fan Yang for helpful conversations on this work.


\end{document}